%% file: dstsaturation.tex
\documentclass{amsart}
\usepackage{proof}
\usepackage{dst}
\usepackage{ast2}

\usepackage[margin=1.4in]{geometry}

\newcommand{\ACzint}{\ensuremath{\usftext{AC}_0^\intern}}

\title{The strength of countable saturation}
\author{Benno van den Berg$^1$}
\address{${}^1$ Institute for Logic, Language and Computation, Universiteit van Amsterdam, P.O. Box 94242, NL, 1090 GE Amsterdam, the Netherlands. E-mail: bennovdberg@gmail.com.}
\author{Eyvind Briseid$^2$}
\address{${}^2$ GFU/LUI, Oslo and Akershus University College of Applied Sciences, PO box 4 St.~Olavs plass, NO-0130 Oslo, Norway. E-mail: eyvindbriseid@gmail.com. Supported by the Research Council of Norway (Project 204762/V30).}
\author{Pavol Safarik$^3$}
\address{$^3$ Fachbereich Mathematik, Technische Universit\"{a}t Darmstadt,
Schlo{\ss}gartenstra{\ss}e 7, 64289 Darmstadt. E-mail:  pavol.safarik@googlemail.com. Supported by the German Science Foundation (DFG project KO 1737/5-1).}
\date{19 May 2016}

\begin{document}

\begin{abstract}
\noindent We determine the proof-theoretic strength of the principle of countable saturation in the context of the systems for nonstandard arithmetic introduced in our earlier work.
\end{abstract}

\maketitle

\input{introduction}
\input{formalities}
\input{csat_intuitionistically}

\input{csat_classically}

\input{reduction}
\input{conclusion}

\bibliographystyle{plain} \bibliography{dst}

\end{document}

%% file: introduction.tex
\section{Introduction}

In \cite{bergbriseidsafarik12} we introduced two systems for nonstandard analysis, one based on classical logic and one on intuitionistic logic. Our aim was to have systems in which one can formalise large parts of nonstandard analysis, which are conservative over well-established standard systems and which allow one to extract computational information from nonstandard proofs.

We analysed various nonstandard principles, but an important principle which we did not discuss in any great detail was the principle of countable saturation:
\[
\CSAT: \quad \forallst n^0 \, \exists x^\sigma \, \Phi(n, x) \to \exists f^{0 \to \sigma} \, \forallst n^0 \, \Phi(n, f(n)).
\]
One reason why this principle is important is that it is involved in the construction of Loeb measures, an often used technique in nonstandard analysis. What we did say is that the principle can be proved in the intuitionistic system introduced in \cite{bergbriseidsafarik12}, while it adds greatly to the proof-theoretic strength of the classical system. The purpose of this short paper is to prove the first claim and to show that the addition of countable saturation to our classical system gives it the proof-theoretic strength of full second-order arithmetic. To show the latter we will give an interpretation of full second-order arithmetic in our classical system extended with countable saturation and rely on earlier work of Escard\'o and Oliva \cite{escardooliva14} to interpret countable saturation using Spector's bar recursion.

%% file: formalities.tex
\section{Formalities}

In this paper we will work with extensions of the system $\eha$ of extensional Heyting arithmetic in all finite types. There are several variants of this system differing, for example, in the way they treat equality. For our purposes decidability of the atomic formulas is not important, so in this respect all variants are equally good. But for the reader who would like to see things fixed we could say we work with a version in which only equality of natural numbers is primitive and equality at higher types is defined extensionally. Also, we could have product types as a primitive notion or not; both options have their advantages and disadvantages, but for us it turns out to be more convenient to not have them as a primitive notion, so that we end up with the system $\eha$ as formalised in \cite[Section 3.3]{kohlenbach08} (this is the system called $\ehazero$ in \cite{troelstra73} and $\ehaarrow$ in \cite{troelstravandalen88b}). The price we have to pay, however, is that we often end up working with tuples of terms and variables of different types and we will have to adopt some conventions for how these ought to be handled. Fortunately, there are some standard conventions here which we will follow (see \cite{kohlenbach08,troelstra73} or \cite{bergbriseidsafarik12}).

What will be important for us, is that $\eha$ is able to handle finite sequences of objects of the same type (not to be confused with the metalinguistic notion of tuple from the previous paragraph). There are at least two ways of doing this: we could extend $\eha$ with types $\sigma^*$ for finite sequences of objects of type $\sigma$, add constants for the empty sequence and the operation of prepending an element to a finite sequence, as well as a list recursor satisfying the expected equations (as in \cite{bergbriseidsafarik12}). Alternatively, we could exploit  the fact that one can code finite sequences of objects of type $\sigma$ as a single object of type $\sigma$ in such a way that every object of type $\sigma$ codes a finite sequence (as in \cite{bergbriseid13}). Moreover, the standard operations on sequences (such as extracting their length or concatenating them) are given by terms in G\"odel's $\T$. For the purposes of this paper, it does not really matter what we do. But whether it is a genuine new type or just syntactic sugar, we will use the notation $\sigma^*$ for finite sequences of objects of type $\sigma$.

In fact, for us finite sequences are really stand-ins for finite sets. For this reason we will often use set-theoretic notation, such as $\emptyset$ for the empty sequence, $\cup$ for concatenation and $\{ x \}$ for the finite sequence of length 1 whose sole component is $x$. And for $x$ of type $\sigma$ and $y$ of type $\sigma^*$ we will write $x \in y$ if $x$ equals one of the components of the sequence $y$.

It remains to define the system $\ehast$ from \cite{bergbriseidsafarik12}. The language of $\ehast$ is obtained from that of $\eha$ by adding unary predicates $\st^\sigma$ as well as two new quantifiers $\forallst x^\sigma$ and $\existsst x^\sigma$ for every type $\sigma \in \T$. Formulas in the old language of $\eha$ (so those not containing these new symbols) we will call \emph{internal}; in contrast, general formulas from $\ehast$ will be called \emph{external}. We will adopt the following
\begin{quote}
{\large \sc{Important convention:}} We follow Nelson \cite{nelson88} in using small Greek letters to denote internal formulas and capital Greek letters to denote formulas which can be external.
\end{quote}

The system $\ehast$ is obtained by adding to $\eha$ the axioms $\EQ, \Tst$ and  $\IA^{\st{}}$, where
\begin{itemize}
\item $\EQ$ stands for the defining axioms of the external quantifiers:
\begin{eqnarray*}
\forallst x \, \Phi(x) & \leftrightarrow &  \forall x \, (\, \st(x)\rightarrow\Phi(x) \, ),\\
\existsst x \, \Phi(x) & \leftrightarrow & \exists x \, (\, \st(x)\wedge\Phi(x) \, ).
\end{eqnarray*}
\item $\Tst$ consists of:
\begin{enumerate}
\item the axioms $\st(x) \land x = y \to \st(y)$,
\item the axiom $\st(t)$ for each closed term $t$ in $\T$,
\item the axioms $\st(f)\wedge\st(x)\rightarrow\st(fx)$.
\end{enumerate}
\item $ \IA^{\st{}}$ is the external induction axiom:
\[
\IA^{\st{}} \quad : \quad\big(\Phi(0)\wedge\forallst x^0 (\Phi(x)\rightarrow\Phi(x+1) )\big)\rightarrow\forallst x^0 \Phi(x).
\]
\end{itemize}
In $\EQ$ and $\IA^{\st{}}$, the expression $\Phi(x)$ is an arbitrary external formula in the language of $\ehast$, possibly with additional free variables. Besides external induction in the form of $\IA^{\st{}}$, the system $\ehast$ also contains the internal induction axiom
\[ \varphi(0) \land \forall x^0 \, ( \, \varphi(x) \to \varphi(x+1) \, ) \to \forall x^0 \, \varphi(x), \]
simply because this is part of $\eha$. But here it is to be understood that this principle applies to internal formulas only. Of course, the laws of intuitionistic logic apply to all formulas of $\ehast$.

It is easy to see that $\ehast$ is a conservative extension of $\eha$: one gets an interpretation of $\ehast$ in $\eha$ by declaring everything to be standard. For more information on $\ehast$, we refer to \cite{bergbriseidsafarik12}. 

%% file: csat_intuitionistically.tex
\section{Countable saturation is weak, intuitionistically}

The purpose of this section is to show that $\CSAT$ does not increase the proof-theoretic strength of the intuitionistic systems for nonstandard arithmetic considered in \cite{bergbriseidsafarik12}. Our main tool for showing this is the $\D$-interpretation from \cite{bergbriseidsafarik12}. We recall its salient features.

The $\D$-interpretation associates to every formula $\Phi(\tup a)$ in the language of $\ehast$ a new formula
\[ \Phi(\tup a)^\D :\equiv \existsst \tup x \, \forallst \tup y \, \varphi_{\D}(\tup x, \tup y, \tup a) \]
where all variables in the tuple $\tup x$ are of sequence type. We do this by induction on the structure of $\Phi(\tup a)$. If $\Phi(\tup a)$ is an atomic formula, then we put
   \begin{itemize}
      \item[(i)] $\varphi(\tup{a})^{\D}:\equiv \varphi_{\D}(\tup{a}):\equiv \varphi(\tup{a})$ if $\Phi(\tup{a})$ is internal atomic formula $\varphi(\tup a)$,
      \item[(ii)] $\st^{\sigma}(u^{\sigma})^{\D}:\equiv \existsst x^{\sigma^*} u \in_{\sigma} x$.
    \end{itemize}
If  $\Phi(\tup{a})^{\D}\equiv\existsst \tup{x}\forallst \tup{y} \, \varphi_{\D} (\tup{x},\tup{y},\tup{a})$ and $\Psi(\tup{b})^{\D}\equiv\existsst \tup{u}\forallst \tup{v} \, \psi_{\D} (\tup{u},\tup{v},\tup{b})$, then
    \begin{itemize}
      \item[(iii)] $(\Phi (\tup{a}) \land \Psi (\tup{b}))^{\D} :\equiv \existsst \tup{x},\tup{u} \forallst \tup{y},\tup{v} \, \big(\varphi_{\D} (\tup{x},\tup{y},\tup{a}) \land \psi_{\D} (\tup{u},\tup{v},\tup{b})\big),$
      \item[(iv)] $(\Phi (\tup{a}) \lor \Psi (\tup{b}))^{\D}  :\equiv \existsst \tup{x},\tup{u} \forallst \tup{y},\tup{v} \, \big(\varphi_{\D} (\tup{x},\tup{y},\tup{a}) \lor \psi_{\D} (\tup{u},\tup{v},\tup{b})\big),$

      \item[(v)] $(\Phi(\tup{a}) \to \Psi(\tup{b}))^{\D}  :\equiv   \existsst \tup{U},\tup{Y} \forallst \tup{x},\tup{v} \, \big(\forall \tup{y} \in \tup{Y}[\tup{x},\tup{v}]\, \varphi_{\D} (\tup{x},\tup{y}, \tup{a})    \to\psi_{\D} (\tup{U}[\tup{x}],\tup{v}, \tup{b})\big).$
      \end{itemize}
In the last line we have used $Y[x]$ as an abbreviation
for
\[ Y[x] := \bigcup_{y \in Y} y(x). \]
This can be regarded as a new application operation, whose associated $\lambda$-abstraction is given by
\[ \Lambda x. t(x) := \{ \lambda x. t(x) \} \]
(for then $(\Lambda x. t(x))[s] = t(s)$).

It remains to consider the quantifiers. For that, assume $\Phi(z,\tup{a})^{\D}\equiv\existsst \tup{x}\forallst \tup{y} \, \varphi_{\D} (\tup{x},\tup{y},z,\tup{a})$, with the free variable $z$ not occuring among the $\tup{a}$. Then
    \begin{itemize}
 \item[(vi)] $(\forall z\Phi(z,\tup{a}))^{\D} :\equiv \existsst \tup{x} \forallst \tup{y} \forall z \, \varphi_{\D} (\tup{x},\tup{y},z,\tup{a}),$
 \item[(vii)] $(\exists z\Phi(z,\tup{a}))^{\D} :\equiv \existsst \tup{x} \forallst \tup{y} \exists z \forall \tup{y'} \in \tup{y}\, \varphi_{\D} (\tup{x},\tup{y'},z,\tup{a}),$
 \item[(viii)] $(\forallst z\Phi(z,\tup{a}))^{\D} :\equiv \existsst \tup{X} \forallst z,\tup{y}  \, \varphi_{\D} (\tup{X}[z],\tup{y},z,\tup{a}),$
 \item[(ix)] $(\existsst z\Phi(z,\tup{a}))^{\D} :\equiv \existsst \tup{x},z \, \forallst \tup{y} \, \exists z' \in z \, \forall \tup y' \in \tup{y}\, \varphi_{\D} (\tup{x},\tup{y'},z',\tup{a}).$
\end{itemize}

We will write $\HH$ for $\ehast$ together with the schema $\Phi \leftrightarrow \Phi^\D$, where $\Phi$ can be any external formula.

\begin{rema}{altdefofHH}
Alternatively, we could define $\HH$ as
\[
     \HH := \ehast + \I + \NCR + \HAC + \HIP_{\forallst} + \HGMP.
\]
See \cite{bergbriseidsafarik12} for a definition of these principles and a proof of this fact.
\end{rema}

 The main result of \cite{bergbriseidsafarik12} on $\HH$ and the $\D$-interpretation was:
\begin{theo}{maintheoremonDst}
If $\HH \vdash \Phi$ and
\[ \Phi^\D :\equiv \existsst \tup x \, \forallst \tup y \, \varphi_\D(\tup x, \tup y), \]
then there is a sequence $\tup t$ of terms from G\"odel's $\T$ such that \[ \eha \vdash \forall \tup y \, \varphi_D(\tup t, \tup y). \] Since $\varphi^\D \equiv \varphi$ for internal $\varphi$, this implies that $\HH$ is a conservative extension of $\eha$.
\end{theo}

The aim of this section is to prove that $\HH \vdash \CSAT$. Before we can do that, we first need to observe that $\eha$ proves a version of the ``finite axiom of choice''.

\begin{lemm}{FACinEHA}
$\eha$ proves that
\[
     \forall s^{0^*} \, \big( \, \forall n \in s \, \exists x^{\sigma} \, \psi(n,x) \to \exists f^{0 \to \sigma} \, \forall n \in s \, \psi (n,f(n)) \, \big).
\]
\end{lemm}
\begin{proof}
By induction on the length $|s|$ of the sequence $s$. Assume $\forall n \in s \, \exists x^{\sigma} \, \psi(n,x)$.
\begin{enumerate}
\item If $|s|= 0$, then any function $f$ will do.
\item If $|s|= k+1$, then write $s_i$ for the $i$th component of $s$ (where $s_0$ is the first and $s_k$ is the last) and $t = \langle s_0, \ldots, s_{k-1} \rangle$ for the sequence obtained from $s$ by deleting the last entry. By induction hypothesis, there is a function $f_0$ such that $\forall n \in t \, \psi(n, f_0(n))$. There are two possibilities:
        \begin{enumerate}
        \item There is $j \lt k$ such that $s_k= s_j$. Then $f_0$ also works for $s$.
        \item For all $j \lt k$ we have $s_k \not= s_j$. Then choose $x_0$ such that $\psi(s_k, x_0)$ and let
        \[
         f(n)=\left\{ \begin{array}{ll} x_0 & \mbox{if } n = s_k, \\
                                                              f_0(n)   & \mbox{else.}
                                     \end{array}
                            \right.
        \]
        \end{enumerate}
\end{enumerate}
\end{proof}

Note that the use of the decidability of equality of objects of type 0 in the previous lemma was necessary in view of the following observation:

\begin{lemm}{FACanddecofeq}
In $\eha$ the finite axiom of choice for sequences of objects of type $\sigma$
\[ \forall s^{\sigma^*} \, \big( \, \forall x \in s \, \exists y^{\tau} \, \psi(x,y) \to \exists f^{\sigma \to \tau} \, \forall x \in s \, \psi (x,f(x)) \, \big) \]
is equivalent to the decidability of the equality of objects of type $\sigma$.
\end{lemm}
\begin{proof}
If the equality of objects of type $\sigma$ is decidable, then the finite axiom of choice for sequences of objects of type $\sigma$ can be argued for as in \reflemm{FACinEHA}. Conversely, suppose that this finite axiom of choice holds and let $a$ and $b$ be two objects of type $\sigma$ and $s$ be the sequence $<a, b>$. Then $\forall x \in s \, \exists n^0 \, ( \, x = s_n)$, so by the finite axiom of choice we have a function $f: \sigma \to 0$ such that
\[ \forall x \in s \, x = s_{f(x)}. \]
We can now decide whether $f(a)$ and $f(b)$ are equal or not, as these are natural numbers. If $f(a) \not= f(b)$, then $a$ and $b$ cannot be equal. If, on the other hand, $f(a) = f(b)$, then $a = s_{f(a)} = s_{f(b)} = b$.
\end{proof}

\begin{theo}{DstrealizersforCSAT}
The theory $\HH$ proves $\CSAT$.
\end{theo}
\begin{proof}
Recall \[ \CSAT \equiv \forallst n^0 \, \exists x^\sigma \, \Phi(n, x) \to \exists f^{0 \to \sigma} \, \forallst n^0 \, \Phi(n, f(n)). \]
So suppose \[ \big(\Phi(n,x)\big)^{\D}\equiv \existsst \tup{u} \, \forallst \tup{v} \, \varphi (\tup{u},\tup{v},n,x). \]
Then
\[
         \big(   \, \forallst n^0 \, \exists x \, \Phi(n,x) \, \big)^{\D}\equiv \existsst \tup{U} \, \forallst n^0,\tup{w} \, \exists x \, \forall \tup v \in \tup w \, \varphi (\tup{U}[n],\tup v,n,x)
\]
and
\[
   \big(   \, \exists f \, \forallst n^0 \, \Phi(n,f(n)) \, \big)^{\D}  \equiv
    \existsst \tup{\tilde{U}} \, \forallst s, \tup{\tilde{w}} \, \exists f \, \forall \tilde{n}^0 \in s, \tilde{\tup v} \in \tup{\tilde{w}}  \, \varphi (\tup{\tilde{U}}[\tilde{n}],\tup{\tilde{v}},\tilde{n} ,f(\tilde{n})),
\]
so $\CSAT^{\D}$ is
\begin{displaymath}
\begin{array}{c}
          \existsst \tup{\tilde{U}},N, \tup{W} \, \forallst  \tup{U} ,s,\tup{\tilde{w}} \\
            \Big(  \, \forall n^0 \in N [\tup{U},s,\tup{\tilde{w}}]  \, \forall \tup w \in \tup{W}[\tup{U},s,\tup{\tilde{w}}] \, \exists x \, \forall \tup v \in \tup w \, \varphi(\tup U[n], \tup v, n, x) \to \\
\exists f  \, \forall \tilde n^0 \in s, \tup{\tilde v} \in \tup{\tilde{w}} \, \varphi(\tup{\tilde{U}} [ \tup{U}, \tilde{n}] , \tup{\tilde{v}} ,\tilde{n}, f(\tilde{n})) \,
           \Big).
\end{array}
\end{displaymath}
So if we put
\begin{eqnarray*}
    \tup{\tilde{U}}   &:=& \Lambda \,  \tup{U} ,\tilde{n} \, . \, \tup{U}[\tilde{n}], \\
     N   &:=& \Lambda \, \tup{U},s,\tup{\tilde{v}} \, . \, s, \\
      \tup{W}   &:=& \Lambda \, \tup{U},s,\tup{\tilde{w}}   \, . \,  \{ \tup{\tilde{w}} \},
\end{eqnarray*}
then we have to show that $\eha$ proves
\[ \forall n^0 \in s \, \exists x \, \forall \tup v \in \tup{\tilde{w}} \, \varphi(\tup U[n], \tup v, n, x) \to \exists f \, \forall \tilde{n}^0 \in s, \tup{\tilde{v}} \in \tup{\tilde{w}} \, \varphi(\tup U[\tilde{n}], \tup{\tilde v}, \tilde{n}, f(\tilde{n})). \]
But this is an instance of the finite axiom of choice (for $\psi(n^0, x) := \forall \tup v \in \tup{\tilde{w}} \, \varphi(\tup U[n], \tup v, n, x)$), so this follows from \reflemm{FACinEHA}.
\end{proof}

%% file: csat_classically.tex
\section{Countable saturation is strong, classically}

From now on we will only work with classical systems. So let $\epa$ be $\eha$ together with the law of excluded middle and $\epast$ be $\ehast$ together with the law of excluded middle.

The aim of this section is to show that, in contrast to what happens in the intuitionistic case, the principle $\CSAT$ in combination with nonstandard principles is very strong in a classical setting. In fact, we need only a simple of form of overspill
\[ \OS_0: \forallst x^0 \, \varphi(x) \to \exists x^0 \, ( \, \lnot \st(x) \land \varphi(x) \, ) \]
in combination with ``$\CSAT$ for numbers''
\[
\CSAT_0: \quad \forallst n^0 \, \exists x^0 \, \Phi(n, x) \to \exists f^{0 \to 0} \, \forallst n^0 \, \Phi(n, f(n)).
\]
to obtain a theory which has at least the strength of second-order arithmetic (in the next section we will show that this lower bound is sharp). More precisely, we have:
\begin{theo}{CSATstrong}
The theory $\epast + \OS_0 + \CSAT_0$ interprets full second-order arithmetic.
\end{theo}
\begin{proof}
For convenience, let us write $\PA_2$ for full second-order classical arithmetic. The idea is to interpret the natural numbers in $\PA_2$ as standard natural numbers in $\epast$ and the subsets of $\NN$ in $\PA_2$ as arbitrary (possibly nonstandard) elements of type $0^*$ in $\epast$, where $n \in s$ is interpreted as: $n$ equals one of the entries of the sequence $s$ (as before). Now:
\begin{enumerate}
\item $\epast$ is a classical system, hence classical logic is interpreted.
\item The Peano axioms for standard natural numbers are part of $\epast$, so these are interpreted as well.
\item Full induction is interpreted, because $\epast$ includes the external induction axiom.
\item So it remains to verify full comprehension. For that it suffices to check that for every formula $\Phi(n^0)$ in the language of $\epast$ there is a sequence $s^{0^*}$ such that
\[ \forallst n \, \big( \, n \in s \leftrightarrow \Phi(n) \, \big). \]
First of all, note that we have
\[ \forallst n \, \exists k \, ( \, k = 0 \leftrightarrow \Phi(n) \, ) \]
by classical logic, so by $\CSAT_0$ there is a function $f^{0 \to 0}$ such that
\[ \forallst n \, ( \, f(n) = 0 \leftrightarrow \Phi(n) \, ). \]
It follows easily by external induction that
\[ \forallst k^0 \, \exists s^{0^*} \, \forall n \leq k \, ( \, n \in s \leftrightarrow f(n) = 0 \, ), \]
so $\OS_0$ gives us a sequence $s$ such that for any standard $n$ we have
\[ n \in s \leftrightarrow f(n) = 0 \leftrightarrow \Phi(n), \]
as desired.
\end{enumerate}
\end{proof}

\begin{rema}{comparisonwithNelson}
From the discussion in Chapter 4 of \cite{nelson77} it seems that $\epast + \OS_0 + \CSAT_0$ is a suitable framework for developing Nelson's ``radically elementary probability theory''. In this connection it is interesting to observe that theorems using $\CSAT_0$, which Nelson calls ``the sequence principle'', are starred in \cite{nelson77}, while in \cite{geyer07} the sequence principle is dropped altogether. Proof-theoretically this makes a lot of sense, because while $\epast + \OS_0$ is conservative over $\epa$ (see \reftheo{factsaboutPP} below), the system $\epast + \OS_0 + \CSAT_0$ has the strength of full second-order arithmetic.
\end{rema}

%% file: reduction.tex
\section{The classical strength of countable saturation}

From now on we will work in $\epast$ extended with the principles
\begin{align*}
\I\ &:\quad \forallst x' \, \exists y \, \forall x \in x' \, \varphi(x, y) \to \exists y \, \forallst x \, \varphi(x, y) \quad \mbox{ and }\\
\HAC_\intern\ &:\quad  \forallst x \, \existsst y \, \varphi(x,y) \to \existsst F \, \forallst x \exists y \in F(x) \,  \varphi (x,y).
\end{align*}
For convenience we will abbreviate this theory as $\PP$. Note that $\I$ implies $\OS_0$ (see \cite[Proposition 3.3]{bergbriseidsafarik12}), so $\PP \vdash \OS_0$. The following theorem summarises the most important facts that we established about $\PP$ in \cite{bergbriseidsafarik12}:

\begin{theo}{factsaboutPP}
To any formula $\Phi$ in the language of $\epast$ one can associate one of the form
\[ \Phi^\Sh :\equiv \forallst \tup x \, \existsst \tup y \, \varphi_\Shb(\tup x, \tup y), \]
in such a way that the following hold:
\begin{enumerate}
\item $\Phi$ and $\Phi^\Sh$ are provably equivalent in $\PP$.
\item Whenever $\Phi$ is provable in $\PP$, there are terms $\tup t$ in G\"odel's $\T$ such that
\[ \epa \vdash \forall \tup x \, \exists \tup y \in \tup t(\tup x) \, \varphi_\Shb(\tup x, \tup y). \]
\item $\varphi^\Sh \equiv \varphi$ for internal formulas $\varphi$.
\end{enumerate}
Hence $\PP$ is a conservative extension of $\epa$.
\end{theo}

The aim of this section is to prove that the strength of $\PP$ extended with $\CSAT$ is precisely that of full second-order arithmetic. As we have already shown that $\PP$ extended with $\CSAT$ has at least the strength of full second-order arithmetic, it suffices to show that $\PP + \CSAT$ can be interpreted in a system which has the strength of full second-order arithmetic. We do this by showing that the $\Sh$-interpretation of $\CSAT$ can be witnessed using Spector's bar recursion \cite{spector62}, which has the strength of full second-order arithmetic \cite[p. 370]{avigadfeferman98}. In fact, recent work by Escard\'o and Oliva \cite{escardooliva14} has shown that the $\Sh$-interpretation of
\begin{align*}
\AC_0^{\st}\ &:\quad \forallst n^0 \, \existsst x^\sigma \, \Phi(n,x)\ \rightarrow \existsst f^{0 \to \sigma} \, \forallst n^0 \Phi(n,f(n))
\end{align*}
can be interpreted using bar recursion. So the following argument, which resembles that in Section 5 in  \cite{nelson88}, suffices to establish that $\PP + \CSAT$ has the strength of full second-order arithmetic:

\begin{theo}{firstred} $\PP \vdash \AC_0^{\st} \to \CSAT$.
\end{theo}
\begin{proof}
We work in $\PP + \AC_0^{\st}$ and have to show that
\[ \forallst n^0 \, \exists x^\sigma \, \Phi(n, x) \to \exists f^{0 \to \sigma} \, \forallst n^0 \, \Phi(n, f(n)). \]
In view of \reftheo{factsaboutPP} it suffices to show this in case $\Phi(n, x)$ is of the form $\forallst \tup u \, \existsst \tup v \, \phi(\tup u,\tup v,n,x)$. To keep the notation simple we will ignore tuples and write simply $u$ and $v$. So, in short, it suffices to show that
\begin{align}
\forallst n^0 \, \exists x \, \forallst u \, \existsst v \, \varphi(u,v,n,x) \label{e:s5p1}
\end{align}
implies
\begin{align}
\exists f \, \forallst n^0 \, \forallst u \, \existsst v \, \varphi(u,v,n,f(n)) \label{e:s5p2}.
\end{align}

By $\HAC_\intern$ we get that \eqref{e:s5p1} implies
\[
\forallst n^0 \, \exists x \, \existsst V \, \forallst u \, \exists v \in V(u) \, \varphi(u,v,n,x),
\]
which is logically equivalent to
\[
\forallst n^0 \, \existsst V \, \exists x \, \forallst u \, \exists v \in V(u) \, \varphi(u,v,n,x),
\]
which by $\AC_0^{\st}$ implies that
\begin{align*}
\existsst V \, \forallst n^0 \, \exists x \, \forallst u \, \exists v \in V(n, u) \, \varphi(u,v,n,x). \tag{\ref{e:s5p1}$'$}\label{e:s5f1}
\end{align*}

On the other hand, \eqref{e:s5p2} follows from
\[
\exists f \, \existsst V \, \forallst n^0,u \, \exists v \in V(n,u) \, \varphi(u,v,n,f(n)),
\]
which logically equivalent to
\[
\existsst V \, \exists f \, \forallst n^0,u \, \exists v \in V(n,u) \, \varphi(u,v,n,f(n)).
\]
By $\I$, this follows from
\[
\existsst V \, \forallst s^{0^*},t \, \exists f \, \forall n \in s, u \in t \, \exists v \in V(n, u) \,  \varphi(u,v,n,f(n)).
 \tag{\ref{e:s5p2}$'$}\label{e:s5f2}
\]
Hence it suffices to show that \eqref{e:s5f1} implies \eqref{e:s5f2}.

Now to do so, let some standard $V$ satisfy \eqref{e:s5f1}, so that we have
\begin{align}
\forallst n^0 \, \exists x \, \forallst u \, \exists v \in V(n,u) \, \varphi(u,v,n,x), \label{e:allX}
\end{align}
and fix arbitrary but standard $s^{0^*}$ and $t$. From  \eqref{e:allX} and the fact that the components of a standard finite sequence are again standard (see \cite[Lemma 2.11]{bergbriseidsafarik12}) it follows that
\[ \forall n \in s \, \exists x \, \forall u \in t \, \exists v \in V(n, u) \, \varphi(u, v, n, x) \]
which by \reflemm{FACinEHA} implies that
\[ \exists f \, \forall n \in s, u \in t \, \exists v \in V(n, u) \, \varphi(u, v, n, f(n)), \]
as desired.
\end{proof}

\begin{rema}{reverse}
In \cite{bergbriseid13} it is shown that the principle that we obtain by restricting $\AC_0^{\st}$ to internal formulas only
\[  \ACzint: \quad \forallst n^0 \, \existsst x^\sigma \, \varphi(n, x) \to \existsst f^{0 \to \sigma} \, \forallst n^0 \, \varphi(n, fn) \]
can be interpreted using a weak form of bar recursion (for binary trees). In the presence of this principle the implication in the previous theorem can be reversed, that is:
\begin{prop}{reverseimpl}
$\PP \vdash \CSAT \land \ACzint \to \AC_0^{\st}$.
\end{prop}
\begin{proof}
We work in $\PP + \CSAT + \ACzint$ and need to prove
\[ \forallst n^0 \, \existsst x^\sigma \, \Phi(n,x)\ \rightarrow \existsst f^{0 \to \sigma} \, \forallst n^0 \Phi(n,f(n)). \]
So assume $\forallst n^0 \, \existsst x^\sigma \, \Phi(n,x)$, or, in other words,
\[ \forallst n^0 \, \exists x^\sigma \, ( \, \st(x) \land \Phi(n,x) \, ). \]
Then it follows from $\CSAT$ that there is a (not necessarily standard) function $g: 0 \to \sigma$ such that
\begin{equation} \label{fiets}
 \forallst n^0 \, ( \, \st(g(n)) \land \Phi(n, g(n)) \, ).
\end{equation}
In particular,
\[ \forallst n^0 \, \existsst x^\sigma \, ( \, x = g(n) \, ), \]
so by $\ACzint$ there is a \emph{standard} function $f: 0 \to \sigma$ such that
\[ \forallst n^0 \, f(n) = g(n). \]
But then it follows from (\ref{fiets}) that
\[ \forallst n^0 \, \Phi(n, f(n)), \]
as desired.
\end{proof}
So, over $\PP$, the principle $\AC_0^{\st{}}$ is equivalent to the conjunction of $\ACzint$ and $\CSAT$.
\end{rema}

\begin{rema}{transfer}
We have shown that $\PP + \CSAT$ can be interpreted in $\epa + \BR$, where $\BR$ stands for Spector's bar recursion, which has the strength of second-order arithmetic.  We could also add countable choice
\[ \forall n^0 \, \exists^\sigma x \, \varphi(n, x) \to \exists F^{0 \to \sigma} \, \forall n^0 \, \varphi(n, F(n)) \]
to the interpreting system and still get a system with the strength of second-order arithmetic. If we do this, we can also interpret transfer with numerical parameters, by which we mean
\[ \NPTP \quad : \quad \forallst \tup t \, ( \, \forallst x \, \varphi(x, \tup t) \to \forall x \, \varphi(x, \tup t) \, ), \]
where the only free variables which are allowed to occur in $\varphi$ are $x$ and $\tup t$, and all variables in $\tup t$ are of type 0. (See Theorem 5 and Remark 6 in \cite{bergsanders14}.) This strengthens earlier results from \cite{hensonkaufmannkeisler84}.
\end{rema}

%% file: conclusion.tex
\section{Conclusion}

We have shown that countable saturation is a weak principle in the intuitionistic context, and is even provable in the intuitionistic nonstandard system we introduced in \cite{bergbriseidsafarik12}. It does however add considerably to the strength of the classical systems we considered there. Indeed, by making heavy use of earlier work of Escard\'o and Oliva we could calibrate its precise strength as that of full second-order arithmetic. This confirms a pattern first observed by Henson, Kaufmann and Keisler \cite{hensonkaufmannkeisler84}: also in their work countable saturation had the effect of making their systems, which originally had the strength of arithmetic, as strong as full second-order arithmetic. Their work in \cite{hensonkeisler86} also suggests that the full saturation principle
\[ \SAT :\equiv  \forallst x^\sigma \, \exists y^\tau \, \Phi(x, y) \to \exists f^{\sigma \to \tau} \, \forallst x^\sigma \, \Phi(x, f(x)) \]
should make the classical systems we considered as strong as full higher-order arithmetic. It would be interesting to see if that is true (the work of Awodey and Eliasson might be useful here \cite{eliasson04,awodeyeliasson04}).

In the intuitionistic context $\SAT$ is again quite weak. Indeed, if $\LEM_\intern$ is the Law of Excluded Middle for internal formulas, then $\HH + \LEM_\intern$ proves the finite axiom of choice for sequences by \reflemm{FACanddecofeq} and $\SAT$ by the argument in \reftheo{DstrealizersforCSAT}. Since $\HH + \LEM_\intern$ is conservative over $\epa$ by the main result of \cite{bergbriseidsafarik12}, this shows that $\HH + \SAT$, $\epa$ and $\eha$ have the same proof-theoretic strength. But we were unable to answer the questions whether  $\SAT$ is provable in $\HH$ and whether $\HH + \SAT$ is a conservative extension of $\eha$. 